\pdfoutput=1
\documentclass[12pt]{amsart}

\usepackage[margin=1in,letterpaper]{geometry}
\usepackage{ifpdf}
\ifpdf
  \usepackage[bookmarks=true,%
    colorlinks=true,%
    linkcolor=blue,%
    citecolor=blue,%
    filecolor=blue,%
    menucolor=blue,%
    pagecolor=blue,%
    urlcolor=blue]{hyperref}
\fi

\usepackage{amssymb,url,booktabs,stmaryrd,mdwtab,mathtools,graphicx,color}
\usepackage[all]{xy}

\newtheorem{theorem}{Theorem}
\newtheorem{corollary}[theorem]{Corollary}
\newtheorem{proposition}[theorem]{Proposition}

\theoremstyle{remark}
\newtheorem{remark}[theorem]{Remark}

\theoremstyle{definition}
\newtheorem{definition}{Definition}

\def\R{\mathbb{R}}

\def\B{\mathcal{B}}
\def\G{\mathcal{G}}
\def\L{\mathcal{L}}
\def\T{\mathcal{T}}
\def\K{\mathcal{K}}

\newcommand{\ONW}{\text{\it {O:NW}}}
\newcommand{\ONE}{\text{\it {O:NE}}}
\newcommand{\OSW}{\text{\it {O:SW}}}
\newcommand{\OSE}{\text{\it {O:SE}}}
\newcommand{\XNW}{\text{\it {X:NW}}}
\newcommand{\XNE}{\text{\it {X:NE}}}
\newcommand{\XSW}{\text{\it {X:SW}}}
\newcommand{\XSE}{\text{\it {X:SE}}}

\begin{document}

\title{Grid Diagrams, Braids, and Contact Geometry}

\author[Ng, Thurston]{Lenhard Ng and Dylan Thurston}


\address{Mathematics Department, Duke
University, Durham, NC 27708}
\email{ng@math.duke.edu}

\address {Mathematics Department, Barnard College, Columbia
  University, New York, NY 10027}
\email{dpt@math.columbia.edu}

\begin{abstract}
We use grid diagrams to present a unified picture of braids,
Legendrian knots, and transverse knots.
\end{abstract}

\maketitle


\section{Introduction}

Grid diagrams, also known in the literature as arc presentations, are
a convenient combinatorial tool for studying knots and links in
$\mathbb{R}^3$. Although grid diagrams (or equivalent structures) have been
studied for over a century (\cite{bib:Bru,bib:Cro,bib:Dyn}), they have
recently regained prominence due
to their role in the combinatorial formulation of knot Floer homology
(\cite{bib:MOS,bib:MOST}).

It has been known for some time that grid
diagrams are closely related to contact geometry as well as to braid
theory. Our purpose here is to indicate the extent to which the
relationships are similar. Indeed, braids, like the Legendrian and
transverse knots in contact geometry, can be viewed as certain
equivalence classes of grid diagrams, and we will see that the various
equivalences fit into one single description.  Furthermore, this
description is compatible with the various maps between these objects,
like the transverse knot constructed from a braid.
Much of the picture we
will present has previously appeared, but we believe that the full
picture (especially the part concerning braids) is new.

\begin{definition}
A \textit{grid diagram} with \textit{grid number $n$} is an $n\times
n$ square grid with $n$ $X$'s and $n$ $O$'s placed in distinct
squares, such that each row and each column contains exactly one $X$
and one $O$.
\label{def:grid}
\end{definition}

We will employ the word ``knot'' throughout as shorthand for
``oriented knot or oriented link''.
Then any grid diagram yields a diagram of a knot in a standard way:
connect $O$ to $X$ in each row, connect $X$ to $O$ in each column, and
have the vertical line segments pass over the horizontal ones
(Figure~\ref{fig:52grid}). In addition, one can associate to any grid
diagram not only a topological knot but
also a braid, a Legendrian knot, and a transverse knot.
We will use the following notation:
\begin{align*}
\G &= \{\text{grid diagrams}\} \\
\K &= \{\text{isotopy classes of topological knots}\} \\
\B &= \{\text{isotopy classes of braids modulo conjugation and exchange}\} \\
\L &= \{\text{Legendrian isotopy classes of Legendrian
  knots}\} \\
\T &= \{\text{transverse isotopy classes of transverse
  knots}\}.
\end{align*}
(For definitions, see Section~\ref{sec:defs}.)

In Section~\ref{sec:defs}, we will review maps between these various
sets that fit together into the following commutative diagram:
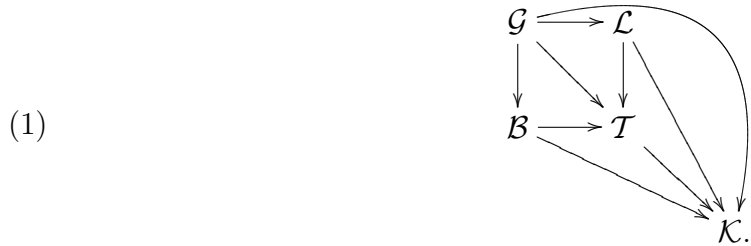
\begin{equation}
\vcenter{\xymatrix{
\G \ar[r] \ar[d] \ar[dr] \ar@/^4pc/[ddrr]
& \L \ar[d] \ar[ddr] & \\
\B \ar[r] \ar[drr] & \T \ar[dr] & \\
& & \K.
}}
\label{eq:diagram}
\end{equation}
Here the map from $\G$ to $\K$ is as described above. For the other maps,  see also \cite{bib:Ben,bib:Cro,bib:Dyn,bib:KN,bib:MM,bib:OST}.

\begin{figure}
\centerline{
\resizebox{5.25in}{!}{\input{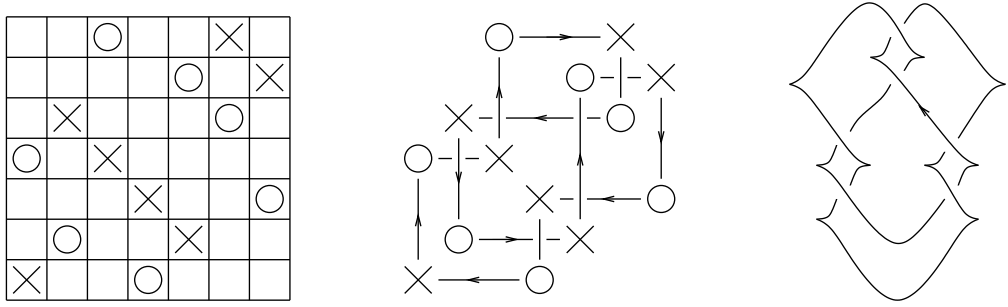}}
}
\caption{A grid diagram and corresponding knot diagram and Legendrian front.}
\label{fig:52grid}
\end{figure}

In \cite{bib:Cro} (see also \cite{bib:Dyn}), Cromwell provides a list
of alterations of grid
diagrams that do not change topological knot type, the grid-diagram
equivalent of Reidemeister moves for knot diagrams. These are
collectively known as \textit{Cromwell moves} and consist of
translations, commutations, and stabilizations/destabilizations. The
last we distinguish into four types, \XNW, \XNE, \XSW, and \XSE,
following \cite{bib:OST}.

\begin{proposition}[Cromwell \cite{bib:Cro}]
The map $\G\rightarrow\K$ sending grid diagrams to topological knots
induces a bijection
\[
\K \longleftrightarrow \G/(\text{translation, commutation, (de)stabilization}).
\]
\end{proposition}

We will see that the maps from $\G$ to $\B$, $\L$, and $\T$ can be
similarly understood. More precisely, we have the following result.

\begin{proposition}
\label{prop:main}
Let $\tilde{\G}$ denote the quotient set $\G/(\text{translation,
  commutation})$.
The maps $\G\rightarrow\B$, $\G\rightarrow\L$, and $\G\rightarrow\T$
induce bijections
\begin{align*}
\B
&\longleftrightarrow \tilde{\G}/(\text{\XNE,\XSE{}
(de)stabilization}) \\
\L &\longleftrightarrow \tilde{\G}/(\text{\XNE,\XSW{} (de)stabilization}) \\
\T &\longleftrightarrow \tilde{\G}/(\text{\XNE,\XSW,\XSE{} (de)stabilization}).
\end{align*}
\end{proposition}

It follows from this result that the maps between $\B,\L,\T,\K$ can
also be interpreted in terms of grid diagrams. For instance, the map
$\B\to\T$ is the quotient
\[
\tilde{\G}/(\text{\XNE,\XSE{} (de)stabilization}) \longrightarrow
\tilde{\G}/(\text{\XNE,\XSW,\XSE{} (de)stabilization}).
\]
Similarly, the maps $\B\to\K$, $\L\to\T$, $\L\to\K$, $\T\to\K$, in
terms of grid diagrams, are quotients by various (de)stabilizations.

\begin{figure}
\centerline{
\resizebox*{!}{1.75in}{\input{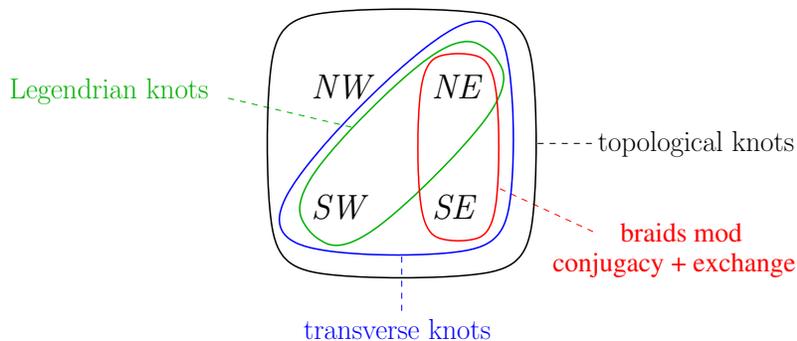}}
}
\caption{
Quotienting $\tilde{\G}$, the set of grid-diagram orbits under
translation and commutation, by various combinations of $X$
(de)stabilizations yields equivalence classes of
braids and various types of knots.
}
\label{fig:stab-diagram}
\end{figure}

Proposition~\ref{prop:main} is summarized diagrammatically in
Figure~\ref{fig:stab-diagram}. The bijections in
Proposition~\ref{prop:main} involving $\L$ and $\T$ have already been
established in \cite{bib:OST}; the new content in this note is the
bijection involving $\B$.

We note that stabilization operations on braids and Legendrian and
transverse knots can be expressed in terms of Cromwell moves. More
precisely, we have the following.

\begin{proposition}
\label{prop:stab}
Under the identifications of Proposition~\ref{prop:main}, we have
\begin{align*}
\text{positive braid stabilization} & \longleftrightarrow
\text{\XSW{} stabilization} \\
\text{negative braid stabilization} & \longleftrightarrow
\text{\XNW{} stabilization} \\
\text{positive Legendrian stabilization} & \longleftrightarrow
\text{\XNW{} stabilization} \\
\text{negative Legendrian stabilization} & \longleftrightarrow
\text{\XSE{} stabilization} \\
\text{transverse stabilization} & \longleftrightarrow
\text{\XNW{} stabilization}.
\end{align*}
\end{proposition}

\noindent
Proposition~\ref{prop:stab} follows from an inspection of the effect
of the various $X$ stabilizations on the corresponding braid or
Legendrian or transverse knot.
See also the table at the end of Section~\ref{ssec:maps}.

Propositions~\ref{prop:main} and~\ref{prop:stab} give
an alternate proof via grid diagrams of the following result.

\begin{proposition}[Transverse Markov Theorem \cite{bib:OSh,bib:Wr}]
\label{prop:TMT}
Two braids represent isotopic transverse knots if and only if they are
related by a sequence of conjugations and positive braid
stabilizations and destabilizations.
\end{proposition}

\noindent
In the usual formulation of Proposition~\ref{prop:TMT}, the map from
braids to transverse knots uses a contact-geometric construction of
Bennequin
\cite{bib:Ben} (cf.\ Section~\ref{ssec:maps}), rather than the map we
use here; see \cite{bib:KN} for a proof that the two maps coincide.

In Section~\ref{sec:defs}, we recall the various relevant constructions
and discuss the effects of grid-diagram symmetries on the maps in
Formula~(\ref{eq:diagram}). We prove our main result,
Proposition~\ref{prop:main}, in Section~\ref{sec:proof}.

\section{Definitions and Maps}
\label{sec:defs}


\subsection{Grid diagrams}

The Cromwell moves on grid diagrams, translation, commutation, and
stabilization/\allowbreak destabilization, are illustrated in
Figure~\ref{fig:52gridmoves} and defined below.  From that
figure it is clear that
each Cromwell move preserves the topological type of the corresponding
knot.

Translation
views a grid diagram as lying on a torus by identifying opposite ends
of the grid, and changes the diagram by translation in the torus. Any
translation is a composition of some number of \textit{vertical
  translations}, which move the top row of the diagram to the bottom
or vice versa, and \textit{horizontal translations}, which move the
leftmost column of the diagram to the rightmost or vice versa.

\begin{figure}
\centerline{
\resizebox{5.25in}{!}{\input{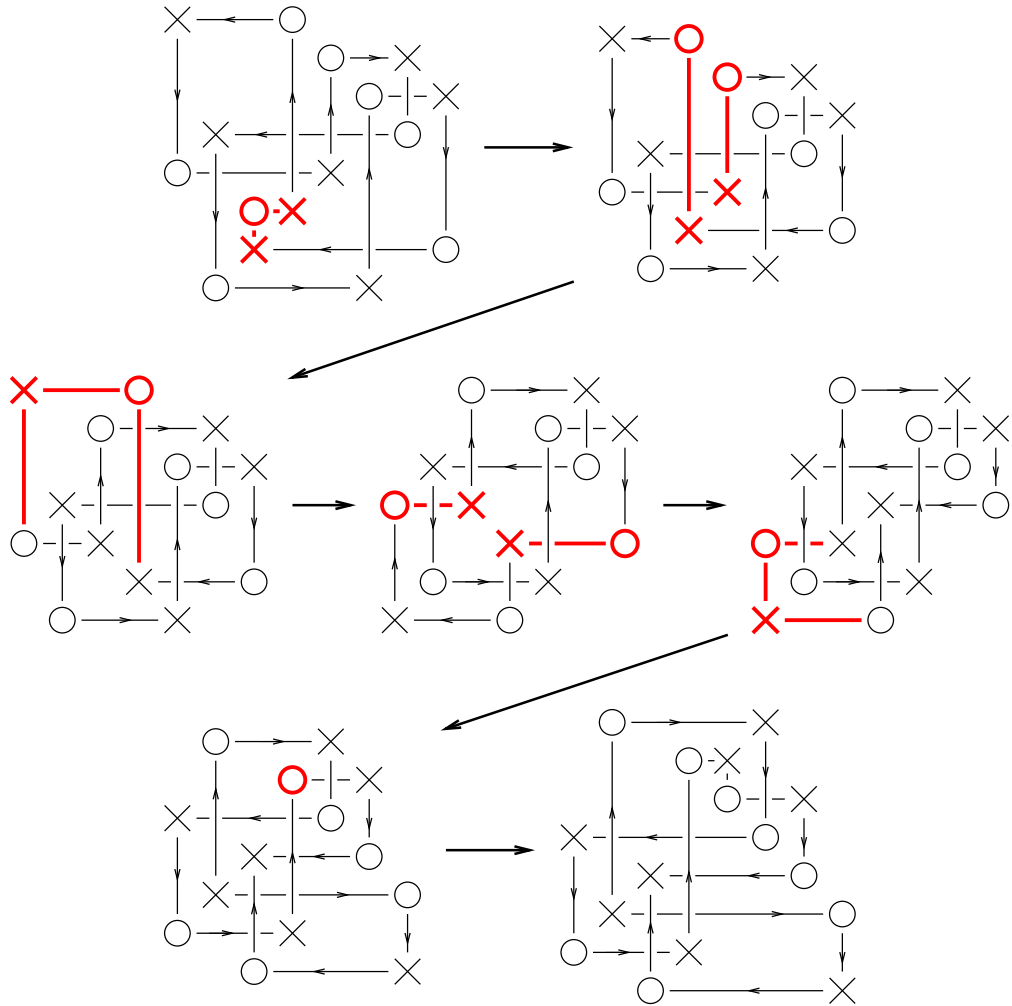}}
}
\caption{Illustration of a sequence of Cromwell moves. In succession:
  \XSE{} destabilization; horizontal commutation; vertical torus
  translation; vertical commutation; horizontal torus translation;
  \OSW{} stabilization. The highlighted sections of each diagram
  indicate the portion that changes under the following move.}
\label{fig:52gridmoves}
\end{figure}

Commutation interchanges two adjacent rows (\textit{vertical
  commutation}) or two adjacent columns (\textit{horizontal
  commutation}). These adjacent rows or columns are required to be
disjoint or nested in the following sense. For rows, the four $X$'s
and $O$'s in the adjacent rows must lie in distinct columns, and the
horizontal line segments connecting $O$ and $X$ in each row must be
either disjoint or nested (one contained in the other) when projected
to a single horizontal line; there is an obvious analogous condition
for columns.

An $X$ (resp.\ $O$) \emph{destabilization}
replaces a $2\times 2$ subgrid containing two $X$'s and one $O$
(resp.\ two $O$'s and one $X$) with a
single square containing an $X$ (resp.\ $O$), eliminating one row and
one column in the process. \emph{Stabilization} is the inverse of
destabilization. Each
(de)stabilization is identified by its type, $X$ or $O$, along with the
corner in the $2\times 2$ subgrid not occupied by a symbol. This
yields eight possibilities: \XNW, \XNE, \XSW,
\XSE, \ONW, \ONE, \OSW, \OSE. It is easy to check that any \ONW{}
(resp.\ \ONE, \OSW, \OSE) (de)stabilization can be expressed as a
composition of translations, commutations, and one \XSE{} (resp.\ XSW,
\XNE, \XNW) (de)stabilization. Thus we restrict our set of Cromwell
moves to include only $X$ (de)stabilizations.

\begin{remark}
  By the argument of \cite[Lemma 4.3]{bib:OST}, we can instead drop
  torus translations 
  and keep matching $O$ (de)stabilizations to yield alternate
  definitions for topological, Legendrian, and transverse knots in
  terms of grid diagrams. In particular, \XNE, \XSW, \OSW, and \ONE{}
  (de)stabilizations, combined with commutations, generate all torus
  translations. The same
  argument can also be adapted for braids: that is, $\B$ is also
  $\G$ modulo commutation and \XNE, \XSE, \ONW, and \OSW{}
  (de)stabilization, as follows.  Sequences of moves similar to those
  from \cite[Lemma 4.3]{bib:OST} show that any horizontal torus
  translation can be achieved by these moves, as can any vertical
  torus translation where the $O$ appears to the left of the~$X$.  But
  any vertical torus translation can be put into the correct
  position by horizontal torus translations.
\end{remark}

\subsection{Braids}

As usual, a \textit{braid} of braid index $n$ is an element of the
group $\B_n$ generated by $\sigma_1,\dots,\sigma_{n-1}$ with relations
$\sigma_i\sigma_{i+1}\sigma_i = \sigma_{i+1}\sigma_i\sigma_{i+1}$ for
$1\leq i\leq n-2$ and $\sigma_i\sigma_j = \sigma_j\sigma_i$ for
$|i-j|\geq 2$. Note the natural inclusion $\B_n \subset \B_{n+1}$
sending $\sigma_i$ to itself for $i \leq n-1$.
The relevant moves to consider on braids are:
\begin{itemize}
\item
braid conjugation: $B \mapsto B'B(B')^{-1}$ for $B,B'\in \B_n$;
\item
exchange move \cite{bib:BM}: $B_1 \sigma_{n-1} B_2 \sigma_{n-1}^{-1} \mapsto B_1
\sigma_{n-1}^{-1} B_2 \sigma_{n-1}$ on $\B_n$, where
$B_1,B_2\in\B_{n-1} \subset \B_n$;
\item
braid stabilization: either positive braid stabilization
$(B \in \B_n) \mapsto (B\sigma_n \in
\B_{n+1})$ or negative braid stabilization
$(B \in \B_n) \mapsto (B\sigma_n^{-1} \in
\B_{n+1})$; and
\item
braid destabilization: the inverse of braid stabilization.
\end{itemize}

In fact, by an observation of Birman and Wrinkle \cite{bib:BW}, an
exchange move can
be expressed as a combination of one positive stabilization, one
positive destabilization, and a number of conjugations. (Here the
positive stabilization and positive destabilization can equally well
be replaced by a negative stabilization and negative destabilization.)
For reference, we include the calculation here.
\begin{align*}
B_1\sigma_{n-1} B_2\sigma_{n-1}^{-1} & \stackrel{\text{conj}}{\longmapsto}
\sigma_{n-1} B_1\sigma_{n-1} B_2 \sigma_{n-1}^{-2}
\stackrel{+ \text{ stab}}{\longmapsto}
\sigma_{n-1} B_1\sigma_{n-1} B_2 \sigma_{n-1}^{-2} \sigma_n \\
& \stackrel{\text{conj}}{\longmapsto}
B_1\sigma_{n-1} B_2 \sigma_{n-1}^{-2} \sigma_n \sigma_{n-1}
= B_1\sigma_{n-1} B_2 \sigma_n \sigma_{n-1} \sigma_n^{-2} \\
&\stackrel{\text{conj}}{\longmapsto}
\sigma_n^{-2} B_1 \sigma_{n-1} \sigma_n B_2 \sigma_{n-1}
= B_1 \sigma_{n-1} \sigma_n \sigma_{n-1}^{-2} B_2 \sigma_{n-1} \\
& \stackrel{\text{conj}}{\longmapsto}
\sigma_{n-1}^{-2} B_2 \sigma_{n-1} B_1 \sigma_{n-1} \sigma_n
\stackrel{+ \text{ destab}}{\longmapsto}
\sigma_{n-1}^{-2} B_2 \sigma_{n-1} B_1 \sigma_{n-1} \\
& \stackrel{\text{conj}}{\longmapsto}
B_1 \sigma_{n-1}^{-1} B_2 \sigma_{n-1}.
\end{align*}

We will depict braids horizontally from left to right, with
strands numbered from top to bottom; for instance, $\sigma_1$
interchanges the top two strands, with the top strand passing over the
other as we move from left to right.

\subsection{Legendrian and transverse knots}

We give a quick description of Legendrian and transverse knots, which
occur naturally in contact geometry; see, e.g., \cite{bib:Et} for more details.
A \textit{Legendrian knot} is a knot in $\R^3$ along which the
standard contact form $dz-y\,dx$ vanishes everywhere; a
\textit{transverse knot} is a knot in $\R^3$ along which $dz-y\,dx>0$
everywhere. (Note for the condition $dz-y\,dx>0$ that the knot is
oriented.) We consider Legendrian (resp.\ transverse) knots up to
\textit{Legendrian isotopy} (resp.\ \textit{transverse isotopy}),
which is simply isotopy through Legendrian (resp.\ transverse) knots.

One convenient way to depict a Legendrian knot is through its
\textit{front
projection}, or projection in the $xz$ plane. A generic front
projection has three features: it has no vertical tangencies;
it is immersed except at cusp singularities; and at all crossings, the
strand of larger slope
passes underneath the strand of smaller slope. Any front with these
features corresponds to a Legendrian knot, with the $y$ coordinate
given by $y=dz/dx$.

The knot diagram corresponding to any grid diagram
can be viewed as the front projection of a Legendrian knot by rotating
it $45^{\circ}$ counterclockwise and smoothing out the corners, creating
cusps where necessary; see Figure~\ref{fig:52grid} for an
example. This yields a map $\G\to\L$ from grid diagrams to isotopy
classes of Legendrian knots. Note that our convention differs from the
convention of \cite{bib:OST}: the convention there is to reverse all
crossings in the grid diagram and then rotate $45^{\circ}$
clockwise. See also Section~\ref{ssec:conventions}.

In \cite{bib:OST}, it is verified that
changing a grid diagram by translation, commutation, or (in our
convention) \XSW, \XNE{}
(de)stabilization does not change the isotopy class of the
corresponding Legendrian knot. Changing by \XNW{} (resp.\ \XSE)
stabilization does change the Legendrian knot type, by
\textit{positive Legendrian stabilization} (resp.\ \textit{negative
  Legendrian stabilization}). Legendrian stabilizations can be
described in the front projection as adding a zigzag, as shown in
Figure~\ref{fig:Legstab}.

\begin{figure}
\centerline{
\resizebox{5in}{!}{\input{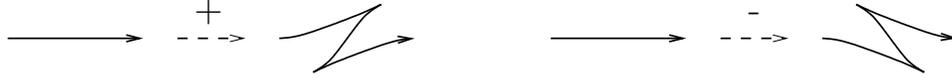}}
}
\caption{
Positive and negative Legendrian stabilizations of the front
projection of a Legendrian knot.
}
\label{fig:Legstab}
\end{figure}

Any Legendrian knot is isotopic to one obtained from some grid
diagram. It is shown in \cite{bib:OST} that the set of equivalence
classes of Legendrian knots under Legendrian isotopy corresponds
precisely to grid diagrams modulo translation, commutation, and \XNE,
\XSW{} (de)stabilization, as presented in Proposition~\ref{prop:main}.

A Legendrian knot can be $C^0$ perturbed to a transverse knot, its
positive transverse pushoff. The resulting map $\L\to\T$ is
not injective; negative Legendrian stabilization does not change the
transverse isotopy type of the positive transverse pushoff. It is a
standard fact in contact geometry \cite{bib:EFM} that this gives a bijection
\[
\T \longleftrightarrow \L/(\text{negative Legendrian stabilization}).
\]
Since negative Legendrian stabilization corresponds to an \XSE{}
Cromwell move, the characterization in Proposition~\ref{prop:main} of
$\T$ as a quotient of $\G$ holds. Note that positive Legendrian
stabilization becomes the ``transverse stabilization'' operation on
transverse knots.

\subsection{Maps between $\G,\B,\L,\T,\K$}
\label{ssec:maps}

Here we collect the constructions of the maps in
Formula~\eqref{eq:diagram}. It suffices to define $\G\to\L$, $\G\to\B$,
$\L\to\T$, $\B\to\T$, and $\T\to\K$, since the other maps follow by
composition. We note that the commutativity of the square
\[
\xymatrix{
\G \ar[r] \ar[d]& \L \ar[d]\\
\B \ar[r] & \T
}
\]
was established in \cite{bib:KN}, and in fact our description of the
maps is essentially identical to the one given there.
The maps $\G\to\L$ and $\L\to\T$ have already been discussed; since the map
$\T\to\K$ is obvious, we are left with $\G\to\B$ and $\B\to\T$.

We begin with the map $\G\to\B$, as described in
\cite{bib:Cro,bib:Dyn}; this is also called a ``flip'' in \cite{bib:MM}.
Any braid in $B_n$ can be viewed as a braid diagram:
a tangle diagram of $n$ strands in the strip $[0,1] \times\R$,
oriented so that the orientation points rightward at all points,
with some collection of $n$ distinct
points $x_1,\dots,x_n\in\R$ for which the braid intersects
$\{0\}\times\R$ and $\{1\}\times\R$ in $\{(0,x_1),\dots,(0,x_n)\}$ and
$\{(1,x_1),\dots,(1,x_n)\}$ respectively. Define a \textit{rectilinear
  braid diagram} (cf.\ ``braided rectangular diagram'' \cite{bib:MM})
to be a tangle diagram in $[0,1]\times\R$ with the
same boundary conditions as a braid diagram, but consisting
exclusively of horizontal and vertical line segments, satisfying the
following properties:
\begin{itemize}
\item
vertical segments always pass over horizontal segments;
\item
each strand can be oriented so that every horizontal segment is oriented
rightwards.
\end{itemize}
Any rectilinear braid diagram can be perturbed into a standard braid
diagram by perturbing vertical segments slightly to point rightwards,
as in Figure~\ref{fig:52braid}.

\begin{figure}
\centerline{
\resizebox{5.25in}{!}{\input{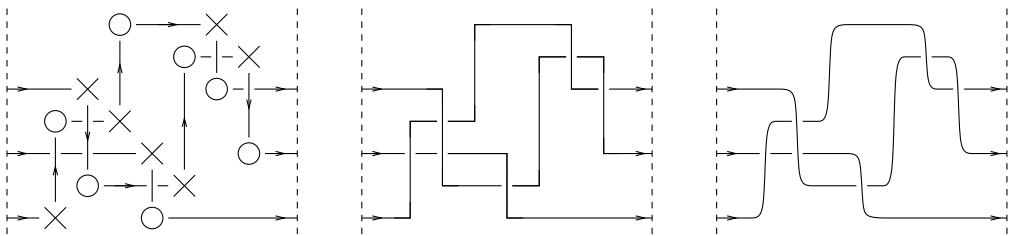}}
}
\caption{
Braid version (left) of the grid diagram in Figure~\ref{fig:52grid}.
Omitting the $X$'s and $O$'s produces a rectilinear braid diagram,
which can be perturbed to become a braid, in this case
$\sigma_2^{-1}\sigma_1\sigma_2^2\sigma_1^2\in\B_3$.}
\label{fig:52braid}
\end{figure}

Now given a grid diagram, one obtains a knot diagram as usual by
drawing horizontal and vertical lines. Turn this into a rectilinear
braid diagram by replacing any horizontal line oriented leftwards from
$O$ to $X$ by two horizontal lines, one pointing rightwards from the $O$,
one pointing rightwards to the $X$, and have these new horizontal lines
pass under all vertical line segments as usual. The rectilinear
braid diagram corresponds to a braid as described above. This produces
the desired map $\G\to\B$.

It remains to define the map $\B\to\T$.
The original contact-geometric definition from \cite{bib:Ben}
is as follows. Identify ends of $B$ to obtain a knot or link in
the solid torus $S^1\times D^2$. View the solid torus as a small
(framed) tubular neighborhood of the standard transverse unknot in $\R^3$ with
self-linking number $-1$. Then $B$ becomes a transverse knot in
a neighborhood of the transverse unknot.

There is also a combinatorial description for the map $\B\to\T$, which
we now describe.  (This description is proven to coincide with the
contact-geometric description in \cite{bib:KN}; see also
\cite{bib:MM,bib:OSh}). Create a
front by replacing each braid crossing as shown in
Figure~\ref{fig:braidfront} and joining corresponding braid
ends. (Joining ends introduces $2n$ cusps for a braid with $n$
strands; see Figure~\ref{fig:braidfront}.) This construction produces
a Legendrian knot from any braid.

\begin{figure}
\centerline{
\resizebox{5.25in}{!}{\input{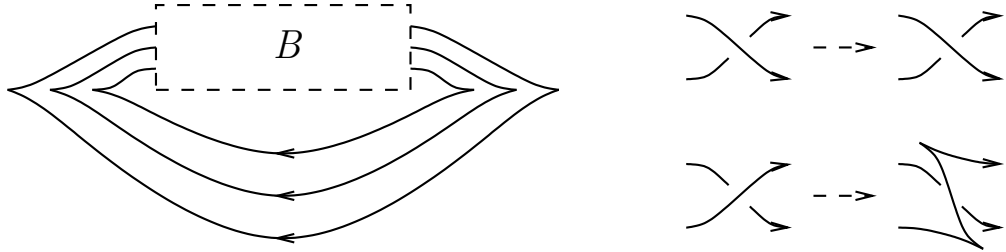}}
}
\caption{
A Legendrian front for a braid $B$.
}
\label{fig:braidfront}
\end{figure}

It is an easy exercise in Legendrian
Reidemeister moves to show that changing the braid by isotopy changes
the Legendrian knot by isotopy and negative Legendrian
(de)stab\-i\-li\-za\-tion; the stabilization is needed when one introduces
cancelling terms $\sigma_i\sigma_i^{-1}$ or $\sigma_i^{-1}\sigma_i$ in
the braid. Similarly, a conjugation or exchange move on a braid
produces a Legendrian isotopy of the Legendrian knot. See
Figure~\ref{fig:braidfrontexch} for the exchange move.

\begin{figure}
\centerline{
\resizebox{5.25in}{!}{\input{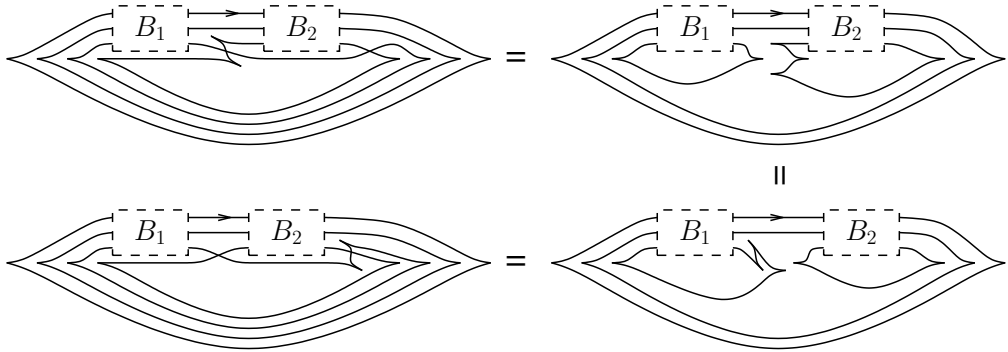}}
}
\caption{
A braid exchange move produces a Legendrian-isotopic front. Equality
denotes Legendrian isotopy.
}
\label{fig:braidfrontexch}
\end{figure}

The map $\B\to\T$ is now given as follows: given a braid, the
corresponding Legendrian front is well-defined up to isotopy and
negative Legendrian stabilization, and hence its positive transverse
pushoff is well-defined. This transverse knot (equivalently, the class
of the Legendrian knot modulo negative Legendrian
(de)stabilization) is unchanged by braid conjugation and exchange.

Table~\ref{tab:moves} has a summary of the effect of the Cromwell
moves on grid diagrams correspond to changes in the associated braid,
Legendrian knot, and transverse knot. The braid column is verified in
Section~\ref{sec:proof}, while the Legendrian and transverse columns
were established in \cite{bib:OST}. For completeness, the table
includes $O$ as well as $X$ stabilizations.

\begin{table}
\begin{center}
\begin{tabular}{@{}llll@{}}
\toprule
Grid diagram & Braid & Legendrian knot & Transverse knot \\
\midrule
torus translation & conjugation & Legendrian isotopy & transverse
isotopy \\
vertical commutation & unchanged & Legendrian isotopy & transverse
isotopy \\
horizontal commutation & conj, exchange & Legendrian isotopy &
transverse isotopy \\
\XNE, \OSW{} stab & unchanged & Legendrian isotopy &
transverse isotopy \\
\XSW, \ONE{} stab & conj, $+$ braid stab &
Legendrian isotopy & transverse isotopy \\
\XSE, \ONW{} stab & unchanged & $-$ Legendrian stab &
transverse isotopy \\
\XNW, \OSE{} stab & conj, $-$ braid stab & $+$
Legendrian stab & transverse stab \\
\bottomrule
\end{tabular}
\end{center}
\caption{The effect of Cromwell moves on associated topological structures.}
  \label{tab:moves}
\end{table}

\subsection{Symmetries and conventions}
\label{ssec:conventions}

Here we discuss various symmetries of grid diagrams and how they
relate the conventions for the maps in Formula (\ref{eq:diagram}) to
other, sometimes conflicting, conventions in the literature.
In this section, we will denote the maps $\G\to\L$,
$\G\to\T$, $\G\to\B$ described in Section~\ref{ssec:maps} by $G\mapsto
L(G)$, $G\mapsto T(G)$, $G\mapsto B_{\shortrightarrow}(G)$,
respectively.

Consider the symmetries $S_1$, $S_2$, $S_3$, and $S_4$ of grid
diagrams defined as follows:
\begin{itemize}
\item
$S_1$ rotates the grid diagram $180^\circ$;
\item
$S_2$ reflects
the diagram about the NE-SW diagonal and interchanges $X$'s and
$O$'s;
\item
$S_3$ reflects the diagram across the horizontal axis; and
\item
$S_4$ rotates the grid diagram $180^\circ$ and interchanges $X$'s and $O$'s.
\end{itemize}
Both $S_1$ and $S_2$ preserve topological knot type, while $S_3$
produces the topological mirror knot $m(K)$ (with reversed orientation
on $\R^3$), and $S_4$ produces the inverse
(i.e., orientation-reversed) knot $-K$.

The symmetries descend to the quotient $\tilde{\G}$ of grid diagrams
by translation and commutation. On $\tilde{\G}$, it is readily checked
that the symmetries permute the four $X$
stabilizations as shown in Table~\ref{tab:symmetries}. We will use this
information to examine the effect of the symmetries on Legendrian and
transverse knots and braids, as shown in the table and explained below.

\begin{table}
  \begin{center}
$\displaystyle \begin{array}{@{}c*{4}{r@{}>{{}}c<{{}}@{}l}c@{}}
\toprule
\text{Symmetry} & \multicolumn{3}{c}{\text{Knot}} & \multicolumn{3}{c}{\text{Braid}}&
\multicolumn{3}{c}{\text{Legendrian}} & \multicolumn{3}{c}{\text{Transverse}} & X \text{ stabilizations} \\ \midrule
S_1 & K &\mapsto& K & B_{\shortrightarrow} &\mapsto& B_{\shortleftarrow} & \hspace{.3em}L &\mapsto& \mu(L) & &\clap{\text{---}}& &
\vcenter{\xymatrix@=3ex{
NW \ar@{<->}[dr] & NE \ar@{<->}[dl] \\ SW & SE
}} \\\addlinespace
S_2 & K &\mapsto& K & B_{\shortrightarrow} &\mapsto& B_{\shortuparrow} & L &\mapsto& L
& T &\mapsto& T &
\vcenter{\xymatrix@=3ex{
NW \ar@(d,l)[] & NE \ar@{<->}[dl] \\ SW & SE \ar@(u,r)[]
}} \\\addlinespace
S_3 & K &\mapsto& m(K) & B_{\shortrightarrow} &\mapsto& m(B_{\shortrightarrow}) &
&\clap{\text{---}}& & &\clap{\text{---}}& &
\vcenter{\xymatrix@=3ex{
NW \ar@{<->}[d] & NE \ar@{<->}[d] \\ SW & SE
}} \\[4ex]
S_4 & K &\mapsto& -K & B_{\shortrightarrow} &\mapsto& -B_{\shortrightarrow} &
L &\mapsto& -\mu(L) & T &\mapsto& -\mu(T) & \vcenter{\xymatrix@=3ex{
NW \ar@(d,l)[] & NE \ar@(d,l)[] \\ SW \ar@(u,r)[] & SE \ar@(u,r)[]
}} \\
\bottomrule
\end{array}$
  \end{center}
  \caption{The effect of symmetries of a grid diagram on associated topological structures.}
  \label{tab:symmetries}
\end{table}

Since $S_1$ and $S_2$ send \XNE, \XSW{} stabilizations to themselves or
each other, Proposition~\ref{prop:main} implies that these symmetries
descend to maps on $\L$. Indeed, it can be shown (see, e.g.,
\cite[Lemma~4.6]{bib:OST}) that $S_2$ does not change Legendrian
isotopy type: $L\circ S_2(G) = L(G)$. It follows also that $T\circ
S_2(G) = T(G)$. On the other hand, we have
$L\circ S_1(G) = \mu(L(G))$, where $\mu :\thinspace \L\to\L$ is the
Legendrian mirror
operation, which reflects Legendrian front diagrams in the horizontal
axis \cite{bib:FT,bib:OST}. In general, the two maps lead to
two distinct Legendrian knots \cite{bib:Ng}; note that Legendrian
``mirroring'' preserves topological type.
We remark that $S_3$ does not descend to a map on $\L$ (there is no
Legendrian version of the topological mirror construction), and
Legendrian mirrors do not descend to the transverse
category.

The map $S_4$ on Legendrian knots produces the orientation reverse of
the Legendrian mirror: $L \mapsto -\mu(L)$.  This operation descends to
(oriented) transverse knots, in an operation that could be called the
transverse mirror.

We next consider braids. Given a grid diagram, there are four
equally valid ways to obtain a map $\G\to\B$ that preserves
topological knot type. One can require that the braid goes from left
to right, as we do in Section~\ref{ssec:maps}, but one could instead
require that the braid go from bottom to top, right to left, or top to
bottom. We write the resulting maps as
$G \mapsto B_{\shortrightarrow}(G)$, $G \mapsto B_{\shortuparrow}(G)$, $G \mapsto
B_{\shortleftarrow}(G)$, and $G \mapsto B_{\shortdownarrow}(G)$,
respectively. In general, these maps lead to four distinct braids,
related by
\[
B_{\shortrightarrow} \circ S_1(G) = B_{\shortleftarrow}(G)
\hspace{6ex} B_{\shortrightarrow} \circ S_2(G) = B_{\shortuparrow}(G)
\hspace{6ex} B_{\shortrightarrow} \circ S_1 \circ S_2(G) =
B_{\shortdownarrow}(G).
\]
As noted in \cite{bib:KN}, it follows from $L\circ S_2(G) = L(G)$ that the
braids $B_{\shortrightarrow}(G)$ and $B_{\shortuparrow}(G)$ represent the same
element of $\T$ even though they usually differ in $\B$, and the same
is true of the pair $B_{\shortleftarrow}(G)$ and $B_{\shortdownarrow}(G)$.
In addition, if we define operations $B \mapsto m(B)$ and $B \mapsto -B$ on
braids, where $m(B)$ replaces every letter in $B$ by its inverse and
$-B$ is the braid word $B$ read backwards, then
$B_{\shortrightarrow} \circ S_3(G) = m(B_{\shortrightarrow}(G))$ and
$B_{\shortrightarrow} \circ S_4(G) = -B_{\shortrightarrow}(G)$.

All symmetries of the NW-NE-SE-SW square are generated by
$S_1,S_2,S_3$. The following generalization of
Proposition~\ref{prop:main} is an immediate consequence of the
symmetries and Proposition~\ref{prop:main}.

\begin{corollary}
We have bijections \label{prop:maincor}
\begin{align*}
\tilde{\G}/(\XNE,\XSE) & \stackrel{B_{\rightarrow}}{\longrightarrow}
\B &
\tilde{\G}/(\XSW,\XSE) & \stackrel{B_{\uparrow}}{\longrightarrow} \B
\\
\tilde{\G}/(\XNW,\XSW) & \stackrel{B_{\leftarrow}}{\longrightarrow} \B
&
\tilde{\G}/(\XNW,\XNE) & \stackrel{B_{\downarrow}}{\longrightarrow} \B
\\
\tilde{\G}/(\XNE,\XSW) & \stackrel{L}{\longrightarrow} \L &
\tilde{\G}/(\XNW,\XSE) & \stackrel{L\circ S_3}{\longrightarrow} \L \\
\tilde{\G}/(\XNE,\XSW,\XSE) & \stackrel{T}{\longrightarrow} \T &
\tilde{\G}/(\XNW,\XNE,\XSW) & \stackrel{T\circ S_1}{\longrightarrow}
\T \\
\tilde{\G}/(\XNW,\XSW,\XSE) & \stackrel{T\circ S_3}{\longrightarrow}
\T &
\tilde{\G}/(\XNW,\XNE,\XSE) & \stackrel{T\circ S_3 \circ
  S_2}{\longrightarrow} \T
\end{align*}
where $L$, $T$ are induced from the maps $\G\to\L$, $\G\to\T$
described in Section~\ref{ssec:maps}.
\end{corollary}

\noindent
Note that three of the bijections in Proposition~\ref{prop:maincor}
involve $S_3$ and thus topological mirroring.

We now discuss the conventions used in Section~\ref{ssec:maps} in
light of symmetries of grid diagrams. Our conventions are chosen to
make the maps in Formula (\ref{eq:diagram}) always preserve
topological knot type. This involves making several choices:
\begin{itemize}
\item
vertical over horizontal line segments in grid diagrams (vs.\
horizontal over vertical), and Legendrian fronts obtained by
$45^\circ$ counterclockwise rotation (vs.\ clockwise);
\item
transverse knots given by positive pushoffs of Legendrian knots (vs.\
negative);
\item
braids going from left to right (vs.\ bottom to top, right to left,
top to bottom).
\end{itemize}

These choices largely agree with the standard conventions in the
literature \cite{bib:Cro,bib:Dyn,bib:EFM,bib:Et,bib:MOS,bib:MOST}. One
can obtain different conventions from ours by applying
grid-diagram symmetries. For braids, this is discussed above, while
for transverse knots, positive pushoffs become negative pushoffs by
applying the symmetry $S_1$: negative pushoffs are transversely isotopic under
\XNW,\XNE,\XSW{} (de)stabilization.

For the knot Floer homology invariant introduced in \cite{bib:OST} and
subsequently used in \cite{bib:KN,bib:NOT}, a slightly different set
of conventions is useful. Here an element $\lambda^+$ of combinatorial
knot Floer homology $\mathit{HK}^-$ is associated to any grid diagram,
and $\lambda^+$
is shown to be invariant under translation, commutation, and
\XNW,\XSW,\XSE{} (de)stabilization. (Another element
$\lambda^-$ is also considered in \cite{bib:OST}; in our notation,
$\lambda^- = \lambda^+ \circ S_1$.) If we apply symmetry $S_2\circ S_3$ to a
grid diagram $G$ before calculating $\lambda^+$, then $\lambda^+$ becomes
an invariant of the transverse knot $T(G)$.

In \cite{bib:KN,bib:NOT,bib:OST}, the map $\G \to \L$ is thus given by $G
\mapsto (L\circ S_2\circ S_3)(G)$ rather than $G \mapsto L(G)$. More explicitly,
given a grid diagram, one can use the horizontal-over-vertical
convention and $45^\circ$ clockwise rotation to obtain a Legendrian
front, as is done in these papers. (In particular, to translate from
our conventions to those of \cite{bib:KN}, first apply $S_2\circ S_3$ to all
grid diagrams.)
Note that due to the presence of
$S_3$, $\lambda^+$ becomes an element of $\mathit{HK}^-$ of the
topological \textit{mirror} of the transverse knot.

\section{Proof of Proposition~\ref{prop:main}}
\label{sec:proof}

Let $B(G)$ ($=B_{\rightarrow}(G)$ from
Section~\ref{ssec:conventions}) denote the braid associated to
a grid diagram $G$ as described in Section~\ref{sec:defs}.
Proposition~\ref{prop:main} (or, more precisely, the braid statement
of Proposition~\ref{prop:main}) is a direct consequence of the following
stronger result.

\begin{proposition} Let $G$ be a grid diagram.
\begin{enumerate}
\item\label{it:1}
Changing $G$ by torus translation or
 \XNE,\XSE{} (de)stabilization changes
$B(G)$ by conjugation.
\item
Changing $G$ by commutation changes $B(G)$ by a combination of
\label{it:2}
conjugation and exchange moves.
\item\label{it:3}
The map $G \mapsto B(G)$ induces a bijection
between
$\G$/(translation, commutation, \XNE, \XSE{} (de)stabilization) and
$\B$/(conjugation, exchange).
\end{enumerate}
\label{prop:gridbraid}
\end{proposition}

\begin{proof}
We first check claims \eqref{it:1} and \eqref{it:2}.
A quick inspection of braid diagrams reveals that changing a grid
diagram $G$ by horizontal commutation or
by \XNE{} or \XSE{} stabilization does not change the braid
isotopy type of $B(G)$.

Changing $G$ by horizontal torus
translation changes $B(G)$ by conjugation; some portion of the
beginning of $B(G)$ is moved to the end, or vice versa. See
Figure~\ref{fig:trans}.

Next we claim that changing $G$ by vertical torus translation also
changes $B(G)$ by conjugation. Indeed, consider moving the topmost
column of $G$ to the bottom. By conjugating by a horizontal torus
translation if necessary, we may assume that in the relevant row,
the $O$ lies to the left of the $X$. Then moving the column keeps the braid
unchanged; see Figure~\ref{fig:trans} again.

\begin{figure}
\centerline{
\resizebox{5.25in}{!}{\input{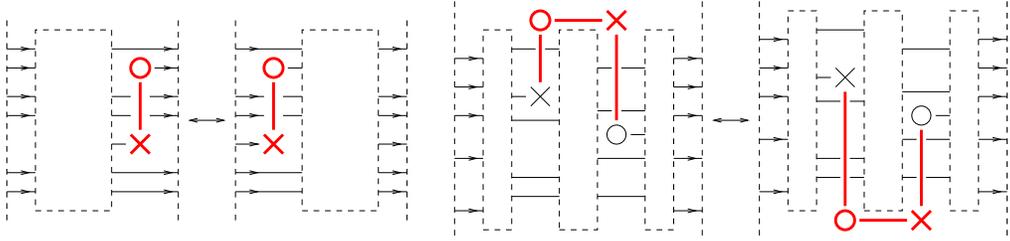}}
}
\caption{The effect on $B(G)$ of changing $G$ by horizontal (left) and
  vertical (right) torus translation. The bold $X$ and $O$ represent
  the column/row being moved.}
\label{fig:trans}
\end{figure}

Finally, we claim that changing $G$ by a vertical commutation
changes $B(G)$ by conjugation and/or
exchange. Indeed, by conjugating with an appropriate torus translation if
necessary, we may assume the following: the two relevant rows are the
bottom two rows in the grid diagram; the $X$ and $O$ in
the bottom row both lie to the right of the $X$ and $O$ in the row
above it; and the bottom right corner of the grid diagram is occupied by
an $X$ or $O$. If $X$ lies to the left of $O$ in both rows, then the
commutation changes $B(G)$ by exchange; otherwise, it does not change
$B(G)$. See Figure~\ref{fig:comm}.

\begin{figure}
\centerline{
\resizebox{5.25in}{!}{\input{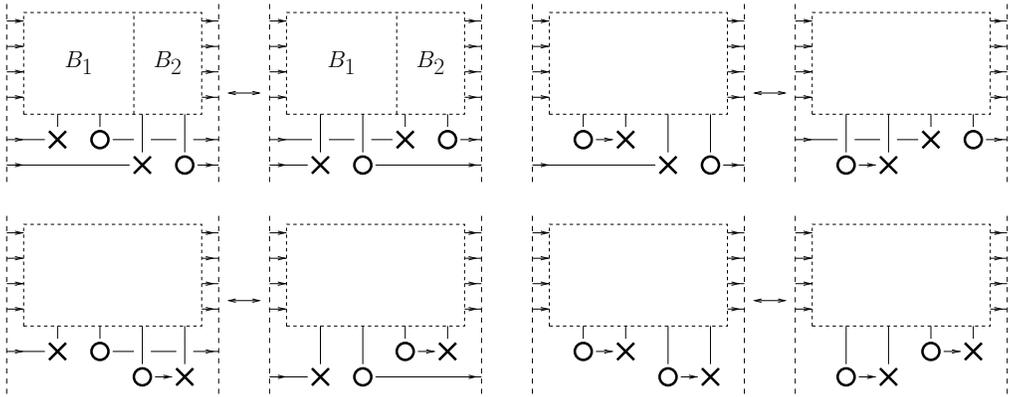}}
}
\caption{The effect on $B(G)$ of changing $G$ by horizontal
  commutation. In three cases, $B(G)$ is unchanged. In the other case
  (upper left), the $n$-strand braid
$B(G)$ changes from $B_1\sigma_{n-1}^{-1}B_2\sigma_{n-1}$ to
  $B_1\sigma_{n-1} B_2\sigma_{n-1}^{-1}$, an exchange move.
}
\label{fig:comm}
\end{figure}

We now establish claim \eqref{it:3}. From claims \eqref{it:1} and
\eqref{it:2}, the
map in (\ref{it:3}) is well-defined. To prove bijectivity, we
construct an inverse. Any braid $B$ can be given a rectilinear braid
diagram by replacing each crossing by an appropriate rectilinear
version; see Figure~\ref{fig:rectbraid}.
\begin{figure}
\centerline{
\resizebox{4in}{!}{\input{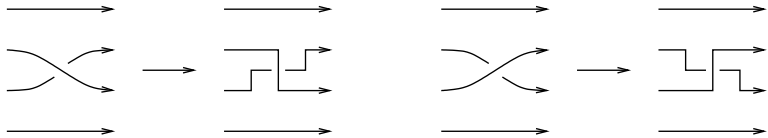}}
}
\caption{Turning a braid diagram into a rectilinear braid diagram.
}
\label{fig:rectbraid}
\end{figure}

Perturb the resulting
rectilinear diagram slightly to another rectilinear diagram for which
no vertical line segments have the
same $x$-coordinate (i.e., are collinear), and no horizontal line
segments have the same $y$-coordinate except for those that are
identified when the ends of the braid are identified.
The perturbed diagram is oriented (from left to right), and each
corner can be assigned an $X$ or $O$ in the usual way. The collection
of $X$'s and $O$'s forms a grid diagram $G(B)$, and by construction we
have $B = B(G(B))$.

\begin{figure}
\centerline{
\resizebox{5.25in}{!}{\input{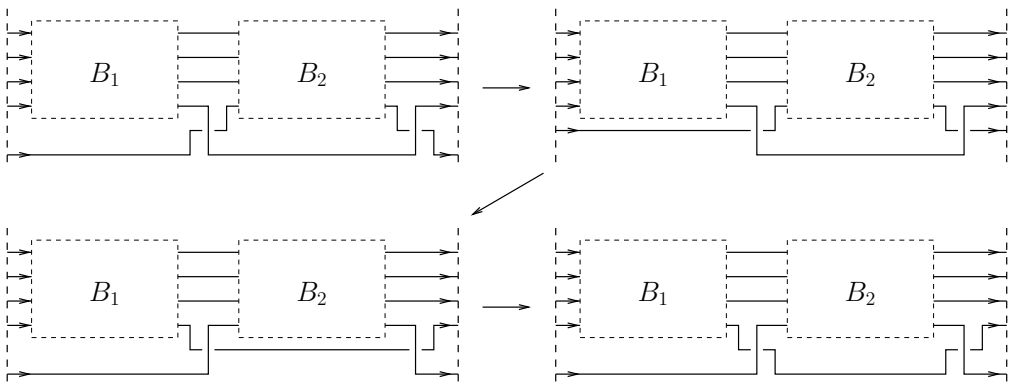}}
}
\caption{Accomplishing an exchange move through a sequence of
  commutation and (de)stabilization moves. The first arrow is given by
  commutations, one \XNE{} destabilization, and one
  \XSE{} destabilization; the second is a horizontal commutation; the
  third is commutations, one \XNE{} stabilization, and one
  \XSE{} destabilization. See also Figure~\ref{fig:braidexch-detail} for
  the moves corresponding to the first and third arrows.
}
\label{fig:braidexch}
\end{figure}

\begin{figure}
\centerline{
\resizebox{4in}{!}{\input{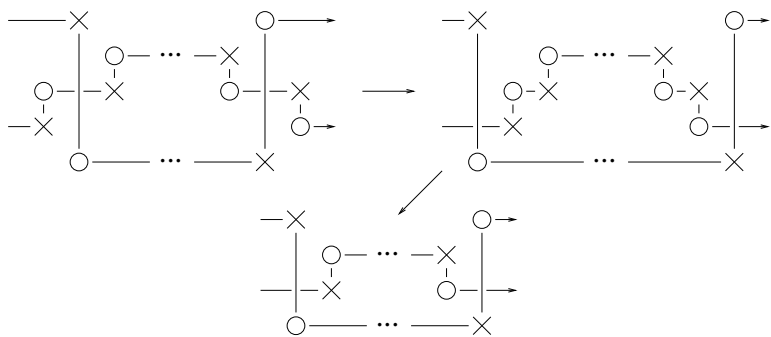}}
}
\caption{
Detail of local moves in the first step of
Figure~\ref{fig:braidexch}. A vertical
commutation move is followed by \XNE{} and \XSE{} destabilization.
}
\label{fig:braidexch-detail}
\end{figure}

Note that $G(B)$ depends on the choice of perturbation from
rectilinear braid diagram to grid diagram, but a different
perturbation simply changes $G(B)$ by commutation. In fact,
up to commutation and \XSW{},\XSE{}
(de)stabilization, $G(B)$ is well-defined for an isotopy class of
braids $B$. This fact is readily established by examining how $G(B)$
changes when the braid word for $B$ changes by one of the relations
$\sigma_i\sigma_i^{-1} = \sigma_i^{-1}\sigma_i = 1$, $\sigma_i\sigma_j
= \sigma_j\sigma_i$ for $|i-j| \geq 2$, and
$\sigma_i\sigma_{i+1}\sigma_i = \sigma_{i+1}\sigma_i\sigma_{i+1}$. See
\cite{bib:Cro} for details.

In addition, changing $B$ by conjugation changes $G(B)$ by
horizontal torus translation, while changing $B$ by an exchange move changes
$G(B)$ by a combination of horizontal commutations and \XNE,\XSE{}
(de)stabilizations; see Figures~\ref{fig:braidexch}
and~\ref{fig:braidexch-detail}.
Thus $B$ induces a map from $\B$/(conjugation, exchange) to
$\G$/(translation, commutation, \XNE, \XSE{} (de)stabilization).

If we consider $G$ and $B$ as maps between $\G$/(translation,
commutation, \XNE, \XSE{} (de)stabilization) and $\B$/(conjugation,
exchange), then as noted earlier, $B \circ G$ is the identity, and one
readily checks that $G \circ B$ is the identity as well. Claim
\eqref{it:3} follows, and the proof of
Proposition~\ref{prop:gridbraid} is complete.
\end{proof}

\section*{Acknowledgments}

LLN thanks the participants of the conference ``Knots in Washington
XXVI'' for useful comments on a preliminary version of the results
presented here.  DPT thanks Ciprian Manolescu, Peter Ozsv\'ath, and
Zolt\'an Szab\'o for helpful conversations.  LLN was supported by NSF grant
DMS-0706777; DPT was supported by a Sloan Research Fellowship.


\end{document}